\documentclass [12pt]{amsart}
\usepackage{amsmath,amssymb}
 \usepackage{amsthm, amsfonts}
 \usepackage{enumerate}

\newtheorem{theorem}{Theorem}[section]

\newtheorem{lemma}[theorem]{Lemma}
\newtheorem{corollary}[theorem]{Corollary}
\theoremstyle{definition}
\newtheorem{definition}[theorem]{Definition}
\newtheorem{example}[theorem]{Example}

\theoremstyle{remark}

\numberwithin{equation}{section}

\begin{document}

\title [{On graded $A$-2-absorbingsubmodules}]{On graded $A$-2-absorbing submodules of graded modules over graded commutative rings }

 \author[{{K. Al-Zoubi and S. Al-kaseasbeh }}]{\textit{Khaldoun Al-Zoubi and Saba Al-Kaseasbeh  }}

\address
{\textit{Khaldoun Al-Zoubi, Department of Mathematics and
Statistics, Jordan University of Science and Technology, P.O.Box
3030, Irbid 22110, Jordan.}}
\bigskip
{\email{\textit{kfzoubi@just.edu.jo}}}

\address
{\textit{Saba Al-Kaseasbeh, Department of Mathematics, Tafila Technical University, Jordan.}}
\bigskip
{\email{\textit{saba.alkaseasbeh@gmail.com}}}

 \subjclass[2010]{13A02, 16W50.}

\date{}
\begin{abstract}
Let $G$ be an  abelian group with identity $e$. Let $R$ be a $G$-graded commutative ring, $M$ a graded $R$-module and $A\subseteq h(R)$ a multiplicatively
closed subset of $R$. In this paper, we introduce the concept of graded $A$-2-absorbing submodules of $M$ as a generalization of graded 2-absorbing
submodules and graded $A$-prime submodules of $M.$ We investigate some
properties of this class of graded submodules.
\end{abstract}

\keywords{graded $A$-2-absorbing submodules, graded $A$-prime submodules, graded 2-absorbing submodules. }
%$*$ Corresponding author}
 \maketitle
%----------------------------------------------------------------------
%----------------------------------------------------------------------
%                Sec1:   Introduction
%----------------------------------------------------------------------
\section{Introduction and Preliminaries}
Throughout this paper all rings are commutative with identity and all
modules are unitary.
Badawi in \cite{16} introduced the concept of 2-absorbing ideals of
commutative rings. The notion of 2-absorbing ideals was extended to
2-absorbing submodules in \cite{18} and \cite{25}.
The concept of $A$-2-absorbing submodules, as a generalization of 2-absorbing
submodules, was introduced in \cite{26} and studied in \cite{19}.

In \cite{27}, Refai and Al-Zoubi introduced the concept of graded primary ideal. The concept of graded 2-absorbing ideals, as a generalization of graded prime ideals, was introduced and studied by Al-Zoubi, Abu-Dawwas and Ceken in \cite{4}. The concept of graded prime submodules was introduced and studied by many authors, see for example \cite{1, 2, 10, 11, 12, 13, 14, 24}.
The concept of graded 2-absorbing submodules, as a generalization of graded prime
submodules, was introduced by Al-Zoubi and Abu-Dawwas in \cite{3} and
studied in \cite{8, 9}. Then many generalizations of graded 2-absorbing submodules were studied such
as graded 2-absorbing primary (see \cite{17}), graded weakly
2-absorbing primary (see \cite{7}) and graded classical 2-absorbing
submodule (see \cite{6}).

\bigskip
Recently, Al-Zoubi and Al-Azaizeh in \cite{5} introduced the concept of
graded $A$-prime submodule over a commutative graded ring as a new
generalization of graded prime submodule.
The main purpose of this paper is to introduced the notion of graded $A$-2-absorbing submodules over a commutative graded ring as a new
generalization of graded 2-absorbing submodules and graded $A$-prime
submodules. A number of results concerning of these classes of graded
submodules and their homogeneous components are given.

First, we recall some basic properties of graded rings and modules which
will be used in the sequel. We refer to \cite{20}, \cite{21}, \cite{22} and
\cite{23} for these basic properties and more information on graded rings
and modules.

Let $G$ be an abelian multiplicative group with identity element $e.$ A ring $R$ is
called \textit{a graded ring (or }$G$\textit{-graded ring) }if there exist
additive subgroups $R_{g}$ of $R$ indexed by the elements $g\in G$ such that
$R=\oplus _{g\in G}R_{g}$ and $R_{g}R_{h}\subseteq R_{gh}$ for all $g,h\in G$. The non-zero elements
of $R_{g}$ are said to be homogeneous of degree $g$ and
all the homogeneous elements are denoted by $h(R)$, i.e. $h(R)=\cup _{g\in
G}R_{g}$. If $r\in R$, then $r$ can be written uniquely as $\sum_{g\in
G}r_{g}$, where $r_{g}$ is called \textit{a homogeneous component of }$r$%
\textit{\ in }$R_{g}$\textit{.} Moreover, $R_{e}$ is a subring of $R$ and $%
1\in R_{e}$ (see \cite{23}).

Let $R=\oplus _{g\in G}R_{g}$ be a $G$-graded ring. An ideal $I$
of $R$ is said to be a \emph{graded ideal} if $I=\oplus _{g\in G}I_{g}$
where $I_{g}=I\cap R_{g}$ for all $g\in G$ (see \cite{23}).

Let $R=\oplus _{h\in G}R_{h}$ be a $G$-graded ring. A left $R$-module $M$ is
said to be \textit{a graded }$R$\textit{-module (or }$G$\textit{-graded }$R$%
\textit{-module)} if there exists a family of additive subgroups $%
\{M_{h}\}_{h\in G}$ of $M$ such that $M=\oplus _{h\in G}M_{h}$ and $%
R_{g}M_{h}\subseteq M_{gh}$ for all $g,h\in G$. Also if an element of $M$
belongs to $\cup _{h\in G}M_{h}=:h(M)$, then it is called \textit{a
homogeneous}. Note that $M_{h}$ is an $R_{e}$-module for every $h\in G$.
Let $M=\oplus _{h\in G}M_{h}$ be a $G$-graded $R$-module. A submodule $%
C$ of $M$ is said to be a \textit{graded submodule of M} if $C=\oplus _{h\in
G}C_{h}$ where $C_{h}=C\cap M_{h}$ for all $h\in G$. In this case, $C_{h}$
is called the $h$\textit{-component of }$C$ (see \cite{23}).

Let $R$ be a $G$-graded ring, $M$ a graded $R$-module, $C$ be a graded
submodule of $M$ and $I$ a graded ideal of $R$. Then $(C:_{R}M)$ is defined as $(C:_{R}M)=\{r\in R:rM\subseteq C\}.$ It is
shown in \cite{14} that if $C$ is a graded submodule of $M$, then $(C:_{R}M)$
is a graded ideal of $R$. The graded submodule $\{m\in M:mI\subseteq C\}$
will be denoted by $(C:_{M}I).$
%---------------------------------------------------------------------

%                Sec2:
%----------------------------------------------------------------------

 \section{RESULTS}

%------------------------------------------------------------------------------
%-----------------------------Definition  2.1------------------------------------
\begin{definition}
 Let $R$ be a $G$-graded ring, $M$ a graded $R$-module, $A\subseteq h(R)$ a multiplicatively closed subset of $R$ and $C$ a
graded submodule of $M$ such that $(C:_{R}M)\bigcap A=\emptyset .$ We say
that $C$ is \textit{a graded }$A$\textit{-2-absorbing submodule of }$M$ if there exists a fixed $a_{\alpha }\in A$ and whenever $
r_{g}s_{h}m_{\lambda }\in C,$ where $r_{g}$, $s_{h}\in h(R)$ and $m_{\lambda
}\in h(M),$ implies that $a_{\alpha }r_{g}s_{h}\in (C:_{R}M)$ or $a_{\alpha
}r_{g}m_{\lambda }\in C$ or $a_{\alpha }s_{h}m_{\lambda }\in C.$ In
particular, a graded ideal $J$ of $R$ is said to be \textit{a graded $A$-2-absorbing ideal} if $J$ is a graded $A$-2-absorbing submodule
of the graded $R$-module $R$.
\end{definition}

%_----------------------------------------------------------------------------------------------------------
%-------------------------Example 2.2 -----------------------------------------------------------------------
Recall from \cite{3} that a proper graded submodule $C$ of a $G$-graded $R$%
-module $M$ is said to be \emph{a graded 2-absorbing submodule of $M$} if whenever $%
r_{g},s_{h}\in h(R)$ and $m_{\lambda }\in h(M)$ with $r_{g}s_{h}m_{\lambda
}\in N$, then either $r_{g}m_{\lambda }\in C$ or $s_{h}m_{\lambda }\in C$ or
$r_{g}s_{h}\in (C:_{R}M)$.

It is easy to see that every graded 2-absorbing submodule $C$ of $M$ with $
(C:_{R}M)\bigcap A=\emptyset $ is a graded $A$-$2$-absorbing submodule. The
following example shows that the converse is not true in general.

\begin{example}
Let $G=(\mathbb{Z},+)$ and $R=(\mathbb{Z},+,.)$. Define $R_{g}=\left\{
\begin{array}{c}
\mathbb{Z}\ \text{\ \ \ \ \ if \ }g=0 \\
0\text{ \ \ otherwise }
\end{array}%
\ \right\} .$ Then $R$ is a $G$-graded ring. Let $M=$ $
%TCIMACRO{\U{2124} }%
%BeginExpansion
\mathbb{Z}
%EndExpansion
\times
%TCIMACRO{\U{2124} }%
%BeginExpansion
\mathbb{Z}
%EndExpansion
_{6}.$ Then $M$ is a $G$-graded $R$-module with $M_{g}=\ \left\{ \
\begin{array}{c}
%TCIMACRO{\U{2124} }%
%BeginExpansion
\mathbb{Z}
%EndExpansion
\times \{\bar{0}\}\text{\ \ \ if }g=0 \\
\{0\}\times
%TCIMACRO{\U{2124} }%
%BeginExpansion
\mathbb{Z}
%EndExpansion
_{6}\text{\ \ \ if }g=1 \\
\{0\}\times \{\bar{0}\}\text{\ \ \ \ otherwise\ }%
\end{array} \right\}
 .$ Now, consider the zero graded submodule $C=\{0\}\times \{\bar{0}\}$ of $M.$
Then $C$ is not a graded 2-absorbing submodule of $M$ since $2\cdot 3\cdot
(0,\bar{1})=(0,\bar{0})\in C$, but $2\cdot 3\notin (C:_{%
%TCIMACRO{\U{2124} }%
%BeginExpansion
\mathbb{Z}
%EndExpansion
}M)=\{0\}$, $\ 2\cdot (0,\bar{1})=(0,\bar{2})\notin C$ and $3\cdot (0,\bar{1}%
)=(0,\bar{3})\notin C.$ Let $A=%
%TCIMACRO{\U{2124} }%
%BeginExpansion
\mathbb{Z}
%EndExpansion
-\{0\}\subseteq
%TCIMACRO{\U{2124} }%
%BeginExpansion
\mathbb{Z}
%EndExpansion
_{0}\subseteq h(R)$ be a multiplicatively closed subset of $R$ and put $%
a_{g}=6\in A.$ However an easy computation shows that $C\ $is a graded $A$%
-2-absorbing submodule of $M.$
\end{example}
%_----------------------------------------------------------------------------------------------------------
%-------------------------Theorem 2.3 -----------------------------------------------------------------------
Let $R$ be a $G$-graded ring, $M$ a graded $R$-module and $A\subseteq h(R)$
be a multiplicatively closed subset of $R.$ Then, $A^{\ast }=\{a_{g}\in h(R):
$ $\frac{a_{g}}{1}$ is a unit of $A^{-1}R\}$ is a multiplicatively closed
subset of $R$ containing $A.$

\begin{theorem}
Let $R$ be a $G$-graded ring, $M$ a graded $R$-module
and $A\subseteq h(R)$ a multiplicatively closed subset of $R.$ Then the
following statements hold:
 \begin{enumerate}[\upshape (i)]
   \item Suppose that $A_{1}\subseteq A_{2}\subseteq h(R)$  are two multiplicatively closed subsets
of $R.$ If $C$ is a graded $A_{1}$-$2$-absorbing\ submodule and $(C:_{R}M)\bigcap A_{2}=\emptyset $, then $C$ is a graded $A_{2}$-2-absorbing submodule of $M$.
   \item A graded submodule $C$ of $M$ is a graded $A$-2-absorbing submodule
of $M$ if and only if it is a graded $A^{\ast }$-2-absorbing submodule of $M.$
 \end{enumerate}
\end{theorem}
\begin{proof}
$(i)$ It is clear.

\bigskip

$(ii)$ Assume that $C$ is a graded $A$-2-absorbing submodule of $M.$ First,
we want to show that $(C:_{R}M)\bigcap A^{\ast }=\emptyset .$ Suppose there exists $r_{g}\in (C:_{R}M)\bigcap A^{\ast }.$ Then $%
\frac{r_{g}}{1}$ is a unit of $A^{-1}R,$ it follows that there
exist $s_{h}\in h(R)$ and $a_{i}\in A$ such that $\frac{r_{g}}{1}\frac{s_{h}%
}{a_{i}}=1.$ Hence $b_{j}a_{i}=b_{j}r_{g}s_{h}$ for some $b_{j}\in A.$ So $\
b_{j}a_{i}=b_{j}r_{g}s_{h}\in (C:_{R}M)\bigcap A,$ a contradiction.
Therefore $(C:_{R}M)\bigcap A^{\ast }=\emptyset .$ Then by $(i)$, we get $C$
is a graded $A^{\ast }$-2-absorbing submodule of $M$ since $A\subseteq
A^{\ast }.$ Conversely, suppose that $C$ is a graded $A^{\ast }$-2-absorbing
submodule of $M.$ Let $r_{g}s_{k}m_{h}\in C$ for some $r_{g},$ $s_{k}\in h(R)
$ and $m_{h}\in h(M).$ Then there is a fixed $a_{i}^{\ast }\in A^{\ast }$ such
that $a_{i}^{\ast }r_{g}s_{k}\in (C:_{R}M)$ or $a_{i}^{\ast }r_{g}m_{h}\in C$
or $a_{i}^{\ast }s_{k}m_{h}\in C.$ Since $\frac{a_{i}^{\ast }}{1}$ is a unit
of $A^{-1}R,$ there exist $t_{j}\in h(R)$ and $a_{k},b_{l}\in A$ such that $%
b_{l}a_{k}=b_{l}a_{i}^{\ast }t_{j}.$ It follows that either $%
(b_{l}a_{k})r_{g}s_{k}=b_{l}a_{i}^{\ast }t_{j}r_{g}s_{k}\in (C:_{R}M)$ or $%
(b_{l}a_{k})r_{g}m_{h}=b_{l}a_{i}^{\ast }t_{j}r_{g}m_{h}\in C$ or $%
(b_{l}a_{k})s_{k}m_{h}=b_{l}a_{i}^{\ast }t_{j}s_{k}m_{h}\in C.$ Therefore, $C
$ is a graded $A$-2-absorbing submodule of $M.$
\end{proof}
%_----------------------------------------------------------------------------------------------------------
%-------------------------Theorem 2.4 -----------------------------------------------------------------------

\begin{theorem}
Let $R$ be a $G$-graded ring, $M$ a graded $R$-module and $A\subseteq h(R)$ a multiplicatively closed subset of $R.$ If $C$ is a graded $A$-2-absorbing submodule of $M$, then $A^{-1}C$ is
a graded 2-absorbing submodule of $A^{-1}M.$
\end{theorem}
\begin{proof}
Assume that $C$ is a graded $A$-2-absorbing submodule of $M.$ Let $\frac{%
r_{g_{1}}}{a_{h_{1}}}\frac{s_{g_{2}}}{a_{h_{2}}}\in $ $h(A^{-1}R)$ and $%
\frac{m_{g_{3}}}{a_{h_{3}}}\in h(A^{-1}M)$ such that $\frac{r_{g_{1}}}{%
a_{h_{1}}}\frac{s_{g_{2}}}{a_{h_{2}}}\frac{m_{g_{3}}}{a_{h_{3}}}\in A^{-1}C.$
Then, there exists $a_{h_{4}}\in A$ such that $%
(a_{h_{4}}r_{g_{1}})s_{g_{2}}m_{g_{3}}\in C.$ As $C$ is a graded $A$%
-2-absorbing submodule of $M,$ there is a fixed $a_{h_{5}}\in A$ such that $%
a_{h_{5}}(a_{h_{4}}r_{g_{1}})s_{g_{2}}\in (C:_{R}M)$ or $%
a_{h_{5}}(a_{h_{4}}r_{g_{1}})m_{g_{3}}\in C$ or $%
a_{h_{5}}s_{g_{2}}m_{g_{3}}\in C.$ Hence, we get either $\frac{r_{g_{1}}}{%
a_{h_{1}}}\frac{s_{g_{2}}}{a_{h_{2}}}=\frac{a_{h_{5}}a_{h_{4}}\
r_{g_{1}}s_{g_{2}}}{a_{h_{5}}a_{h_{4}}\ a_{h_{1}}a_{h_{2}}}\in
A^{-1}(C:_{R}M)\subseteq (A^{-1}C:_{A^{-1}R}A^{-1}M)$ or $\frac{r_{g_{1}}%
}{a_{h_{1}}}\frac{m_{g_{3}}}{a_{h_{3}}}=\frac{a_{h_{5}}r_{g_{1}}m_{g_{3}}}{%
a_{h_{5}}a_{h_{1}}a_{h_{3}}}\in A^{-1}C$ or $\frac{s_{g_{2}}}{a_{h_{2}}}%
\frac{m_{g_{3}}}{a_{h_{3}}}=\frac{a_{h_{4}}s_{g_{2}}m_{g_{3}}}{%
a_{h_{4}}a_{h_{2}}a_{h_{3}}}\in A^{-1}C.$ Therefore, $A^{-1}C$ is a graded
2-absorbing submodule of $A^{-1}M.$
\end{proof}
 %----------------------------------------------------
%_----------------------------------------------------------------------------------------------------------
%-------------------------Lemma 2. 5-----------------------------------------------------------------------

\begin{lemma}
Let $R$ be a $G$-graded ring, $M$ a graded $R$-module, $%
A\subseteq h(R)$ be a multiplicatively closed subset of $R$ and $C$ a graded
$A$-2-absorbing submodule of $M.$ Let $K=\bigoplus\limits_{\lambda \in
G}K_{\lambda }$ be a graded submodule of $M.$ Then there exists a fixed $%
a_{\alpha }\in A$\ and whenever $r_{g}$, $s_{h}$ $\in h(R)$ and $\lambda \in
G$ such that $r_{g}s_{h}K_{\lambda }\subseteq C,$ then $a_{\alpha
}r_{g}K_{\lambda }\subseteq C$ or $\ a_{\alpha }s_{h}K_{\lambda }\subseteq C$
or $a_{\alpha }r_{g}s_{h}\in (C:_{R}M).$
\end{lemma}
\begin{proof}
 Let $r_{g}$, $s_{h}$ $\in h(R)$, and $\lambda \in G$ such that
$r_{g}s_{h}K_{\lambda }\subseteq C.$ Since $C$ is a graded $A$-2-absorbing
submodule of $M,$\ there exists $a_{\alpha }\in A$ so that $%
r_{g}s_{h}m_{\lambda }\in C$\ implies $a_{\alpha }r_{g}s_{h}\in (C:_{R}M)$
or $a_{\alpha }r_{g}m_{\lambda }\in C$ or $a_{\alpha }s_{h}m_{\lambda }\in C$ for each $r_{g},s_{h}\in h(R)$\ and $m_{\lambda }\in h(M).$ Now, we will
show that $a_{\alpha }r_{g}K_{\lambda }\subseteq C$ or $a_{\alpha
}s_{h}K_{\lambda }\subseteq C$ or $a_{\alpha }r_{g}s_{h}\in (C:_{R}M).$
Assume on the contrary that $a_{\alpha }r_{g}K_{\lambda }\nsubseteq C$, $\ a_{\alpha
}s_{h}K_{\lambda }\nsubseteq C$ and $a_{\alpha }r_{g}s_{h}\notin (C:_{R}M).
$ Then there exist $k_{\lambda \text{ }},$ $k_{\lambda }^{\prime }\in K$
such that $a_{\alpha }r_{g}k_{\lambda \text{ }}\notin C$ and $a_{\alpha
}s_{h}k_{\lambda }^{\prime }\notin C.$ Since $C$ is a graded $A$-2-absorbing
submodule of $M,$ $r_{g}s_{h}k_{\lambda \text{ }}\in C,$ $a_{\alpha
}r_{g}k_{\lambda } \notin C$ and $a_{\alpha }r_{g}s_{h}\notin
(C:_{R}M),$ we get $a_{\alpha }s_{h}k_{\lambda \text{ }}\in C$. In a
similar manner, we get $a_{\alpha }r_{g}k_{\lambda }^{\prime }\in C.$ By $%
k_{\lambda}+k_{\lambda }^{\prime }\in K_{\lambda }\subseteq h(M)$
it follows that $r_{g}s_{h}(k_{\lambda}+k_{\lambda }^{\prime
})\in C.$ Since $C$ is a graded $A$-2-absorbing submodule of $M$ and $%
a_{\alpha }r_{g}s_{h}\notin (C:_{R}M),$ we have either $ a_{\alpha
}r_{g}(k_{\lambda}+k_{\lambda }^{\prime })\in C$ or $ a_{\alpha
}s_{h}(k_{\lambda}+k_{\lambda }^{\prime })\in C.$ If $a_{\alpha
}r_{g}(k_{\lambda}+k_{\lambda }^{\prime })=a_{\alpha
}r_{g}k_{\lambda}+a_{\alpha }r_{g}k_{\lambda }^{\prime }\in C,$
then we get $a_{\alpha }r_{g}k_{\lambda \text{ }}\in C$ since $a_{\alpha
}r_{g}k_{\lambda }^{\prime }\in C,$ a contradiction. If $\ a_{\alpha
}s_{h}(k_{\lambda}+k_{\lambda }^{\prime })=a_{\alpha
}s_{h}k_{\lambda}+a_{\alpha }s_{h}k_{\lambda }^{\prime }\in C,$
then we get $a_{\alpha }s_{h}k_{\lambda }^{\prime }\in C$ since $a_{\alpha
}s_{h}k_{\lambda}\in C,$ a contradiction. Therefore $a_{\alpha
}r_{g}K_{\lambda }\subseteq C$ or $ a_{\alpha }s_{h}K_{\lambda }\subseteq C$
or  $a_{\alpha }r_{g}s_{h}\in (C:_{R}M).$

\end{proof}
 %----------------------------------------------------

%_----------------------------------------------------------------------------------------------------------
%-------------------------Theorem 2.6 -----------------------------------------------------------------------

\begin{theorem}
Let $R$ be a $G$-graded ring, $M$ a graded $R$-module, $C$ a graded submodule\emph{\ }of $M$ and $A\subseteq h(R)$ be a multiplicatively closed subset of $R$ with $(C:_{R}M)\bigcap A=\emptyset.$ Let $I=$\ $\bigoplus\limits_{g\in G}I_{g}$ and $J=$\ $\bigoplus\limits_{h\in
G}J_{h}$ be a graded ideals of $R$ and $K=\bigoplus\limits_{\lambda \in
G}K_{\lambda }$ a graded submodule of $M.$ Then the following statements are
equivalent:
 \begin{enumerate}[\upshape (i)]
   \item $C$ is a graded $A$-2-absorbing submodule of $M;$
   \item There exists a fixed $a_{\alpha }\in A$\ such that $
I_{g}J_{h}K_{\lambda }\subseteq C$  for some $g,h,\lambda \in G$ implies
either $a_{\alpha }I_{g}K_{\lambda }\subseteq C$ or $\ a_{\alpha
}J_{h}K_{\lambda }\subseteq C$ or $a_{\alpha }I_{g}J_{h}\subseteq
(C:_{R}M).$
 \end{enumerate}
\end{theorem}
\begin{proof}
$(i)\Rightarrow (ii)$
Assume that $C$ is a graded $A$-2-absorbing submodule of $M$ and $%
g,h,\lambda \in G$ such that $I_{g}J_{h}K_{\lambda }\subseteq C.$ Since $C$ is a graded $A$-2-absorbing submodule of $M,$ there exists a
fixed $a_{\alpha }\in A$\ so that $r_{g}s_{h}m_{\lambda }\in C$\ implies $%
a_{\alpha }r_{g}s_{h}\in (C:_{R}M)$ or $a_{\alpha }r_{g}m_{\lambda }\in C$
or $a_{\alpha }s_{h}m_{\lambda }\in C$ \ for each $r_{g},s_{h}\in h(R)$\ and
$m_{\lambda }\in h(M).$ Now, we will show that $a_{\alpha }I_{g}K_{\lambda
}\subseteq C$ or $\ a_{\alpha }J_{h}K_{\lambda }\subseteq C$ or \ $a_{\alpha
}I_{g}J_{h}\subseteq (C:_{R}M).$ Assume on the contrary that $a_{\alpha
}I_{g}K_{\lambda }\nsubseteq C$, $\ a_{\alpha }J_{h}K_{\lambda }\nsubseteq C$
and \ $a_{\alpha }I_{g}J_{h}\nsubseteq (C:_{R}M).$ Then there exist $%
x_{g}\in I_{g}$ and $y_{h}\in J_{h}$ such that $a_{\alpha }x_{g}K_{\lambda
}\nsubseteq C$ and $a_{\alpha }y_{h}K_{\lambda }\nsubseteq C.$ Since $\
x_{g}y_{h}K_{\lambda }\subseteq C,$ by Lemma 2.5, we get $a_{\alpha
}x_{g}y_{h}\in (C:_{R}M).$ Since \ $a_{\alpha }I_{g}J_{h}\nsubseteq
(C:_{R}M),$ there exist $r_{g}\in I_{g\text{ }}$ and $s_{h}\in J_{h}$ such
that $a_{\alpha }r_{g}s_{h}\notin (C:_{R}M).$ Then by Lemma 2.5, we have $%
a_{\alpha }r_{g}K_{\lambda }\subseteq C$ or $a_{\alpha }s_{h}K_{\lambda
}\subseteq C$ since $r_{g}s_{h}K_{\lambda }\subseteq C.$ Consider the
following three cases:

\bigskip

\textbf{Case1:}  $a_{\alpha }r_{g}K_{\lambda }\subseteq C$ and $a_{\alpha
}s_{h}K_{\lambda }\nsubseteq C.$ Since $x_{g}s_{h}K_{\lambda }\subseteq C,$ $%
a_{\alpha }s_{h}K_{\lambda }\nsubseteq C$ and $a_{\alpha }x_{g}K_{\lambda
}\nsubseteq C,$ by Lemma 2.5, we get $a_{\alpha }x_{g}s_{h}\in (C:_{R}M).$
As $a_{\alpha }x_{g}K_{\lambda }\nsubseteq C$ and $a_{\alpha
}r_{g}K_{\lambda }\subseteq C,$ we have $a_{\alpha }(x_{g}+r_{g})K_{\lambda
}\nsubseteq C.$ By $(x_{g}+r_{g})\in I_{g}$ it follows that $%
(x_{g}+r_{g})s_{h}K_{\lambda }\subseteq C.$ Since $(x_{g}+r_{g})s_{h}K_{%
\lambda }\subseteq C$, $a_{\alpha }(x_{g}+r_{g})K_{\lambda }\nsubseteq C$
and $a_{\alpha }s_{h}K_{\lambda }\nsubseteq C,$ by Lemma 2.5, we get $%
a_{\alpha }(x_{g}+r_{g})s_{h}\in (C:_{R}M).$ By $a_{\alpha
}(x_{g}+r_{g})s_{h}\in (C:_{R}M)$ and $a_{\alpha }x_{g}s_{h}\in (C:_{R}M)\ $%
it follows that $a_{\alpha }r_{g}s_{h}\in (C:_{R}M)$ which is a
contradiction.

\bigskip \textbf{Case2:} $a_{\alpha }r_{g}K_{\lambda }\nsubseteq C$ and $%
a_{\alpha }s_{h}K_{\lambda }\subseteq C.$ The proof is similar to that of
Case 1.

\bigskip \textbf{Case 3:}  $a_{\alpha }r_{g}K_{\lambda }\subseteq C$ and $%
a_{\alpha }s_{h}K_{\lambda }\subseteq C.$ Since $a_{\alpha }y_{h}K_{\lambda
}\nsubseteq C$ and $a_{\alpha }s_{h}K_{\lambda }\subseteq C,$ we get $%
a_{\alpha }(s_{h}+y_{h})K_{\lambda }\nsubseteq C.$ By $(s_{h}+y_{h})\in J_{h%
\text{ }}$ it follows that $x_{g}(s_{h}+y_{h})K_{\lambda }\subseteq C.$
Since $x_{g}(s_{h}+y_{h})K_{\lambda }\subseteq C$, $a_{\alpha
}(s_{h}+y_{h})K_{\lambda }\nsubseteq C$ and $a_{\alpha }x_{g}K_{\lambda
}\nsubseteq C,$ by Lemma 2.5, we get $a_{\alpha }x_{g}(s_{h}+y_{h})\in
(C:_{R}M).$ Then we get $a_{\alpha }x_{g}s_{h}\in (C:_{R}M)$ since $a_{\alpha
}x_{g}(s_{h}+y_{h})\in (C:_{R}M)$ and $a_{\alpha }x_{g}y_{h}\in (C:_{R}M).$
As $a_{\alpha }x_{g}K_{\lambda }\nsubseteq C$ and $a_{\alpha
}r_{g}K_{\lambda }\subseteq C,$ we have $a_{\alpha }(r_{g}+x_{g})K_{\lambda
}\nsubseteq C.$ Then by Lemma 2.5, $a_{\alpha }(r_{g}+x_{g})y_{h}\in
(C:_{R}M)$ since $(r_{g}+x_{g})y_{h}K_{\lambda }\subseteq C,$ $a_{\alpha
}(r_{g}+x_{g})K_{\lambda }\nsubseteq C$ and $a_{\alpha }y_{h}K_{\lambda
}\nsubseteq C.$ Since $a_{\alpha }(r_{g}+x_{g})y_{h}\in (C:_{R}M)$ and $%
a_{\alpha }x_{g}y_{h}\in (C:_{R}M),$ we get $a_{\alpha }r_{g}y_{h}\in
(C:_{R}M).$ Thus by Lemma 2.5, we get $a_{\alpha
}(r_{g}+x_{g})(s_{h}+y_{h})\in (C:_{R}M)$ since $%
(r_{g}+x_{g})(s_{h}+y_{h})K_{\lambda }\subseteq C,$ $a_{\alpha
}(r_{g}+x_{g})K_{\lambda }\nsubseteq C$ and $a_{\alpha
}(s_{h}+y_{h})K_{\lambda }\nsubseteq C.$ As $a_{\alpha
}(r_{g}+x_{g})(s_{h}+y_{h})=a_{\alpha }r_{g}s_{h}+a_{\alpha
}r_{g}y_{h}+a_{\alpha }x_{g}s_{h}+a_{\alpha }x_{g}y_{h}\in (C:_{R}M)$ and $%
a_{\alpha }r_{g}y_{h},$ $a_{\alpha }x_{g}s_{h},$ $a_{\alpha }x_{g}y_{h}\in
(C:_{R}M),$ we have $a_{\alpha }r_{g}s_{h}\in (C:_{R}M),$ a contradiction.

\bigskip

$(ii)\Rightarrow (i)$ Assume that $(ii)$ holds. Let $r_{g}$, $s_{h}\in h(R)$
and $m_{\lambda }\in h(M)$ such that $r_{g}s_{h}m_{\lambda }\in C.$ Let $%
I=r_{g}R\ $and $J=s_{h}R\ $\ be a graded ideals of $R$ generated by $r_{g}$
and $s_{h}$, respectively and $K=m_{\lambda }R$ a graded submodule of $M$
generated by $m_{\lambda }.$ Then $I_{g}J_{h}K_{\lambda }\subseteq C.$ By
our assumption, there exists $a_{\alpha }\in A$ such that either $a_{\alpha
}I_{g}K_{\lambda }\subseteq C$ or $\ a_{\alpha }J_{h}K_{\lambda }\subseteq C$
or \ $a_{\alpha }I_{g}J_{h}\subseteq (C:_{R}M).$ This yields that either $%
a_{\alpha }r_{g}m_{\lambda }\in C$ or $a_{\alpha }s_{h}m_{\lambda }\in C$ or
$a_{\alpha }r_{g}s_{h}\in (C:_{R}M).$ Therefore, $C$ is a graded $A$%
-2-absorbing submodule of $M$.
\end{proof}
 %----------------------------------------------------
 %_----------------------------------------------------------------------------------------------------------
%-------------------------Corollary 2.7  -----------------------------------------------------------------------

\begin{corollary}
Let $R$ be a $G$-graded ring, $P$ a\ graded ideal\ of $R$ and $%
A\subseteq h(R)$ be a multiplicatively closed subset of $R$ with $P\bigcap
A=\emptyset .$ Let $I=$\ $\bigoplus\limits_{g\in G}I_{g}$, $J=$\ $%
\bigoplus\limits_{h\in G}J_{h}$ and $L=\bigoplus\limits_{\lambda \in
G}L_{\lambda }$ be a graded ideals of $R.$ Then the following statements are
equivalent:

 \begin{enumerate}[\upshape (i)]
   \item $P$ is a graded $A$-2-absorbing ideal of $M;$
   \item There exists $a_{\alpha }\in A$\ such that $I_{g}J_{h}L_{\lambda
}\subseteq P$ for some $g,h,\lambda \in G$ implies either $a_{\alpha
}I_{g}L_{\lambda }\subseteq P$ or $a_{\alpha }J_{h}L_{\lambda }\subseteq P$
or $a_{\alpha }I_{g}J_{h}\subseteq P$.

 \end{enumerate}
\end{corollary}

  %_----------------------------------------------------------------------------------------------------------
%-------------------------Theorem 2.8  -----------------------------------------------------------------------
\begin{theorem}
 Let $R$ be a $G$-graded ring, $M$ a graded $R$-module, $C$ a
graded submodule of $M$ and $A\subseteq h(R)$ be a multiplicatively closed
subset of $R.$ If $C$is a graded $A$-2-absorbing submodule of $M,$ then $%
(C:_{R}M)$ is a graded $A$-2-absorbing ideal of $R$.
\end{theorem}
\begin{proof}
Assume that $C$ is a graded $A$-2-absorbing submodule
of $M.$ Let $r_{g}s_{h}t_{\lambda }\in (C:_{R}M)$ for some $r_{g}, s_{h}, t_{\lambda }\in h(R).$ Let $I=r_{g}R\ $and $J=s_{h}R\ $\ be a graded
ideals of $R$ generated by $r_{g}$ and $s_{h}$, respectively and $%
K=t_{\lambda }M$ a graded submodule of $M$. Then $I_{g}J_{h}K_{\lambda }\subseteq C.$ By Theorem 2.6, we have a
fixed $a_{\alpha }\in A$ such that $a_{\alpha }I_{g}J_{h}\subseteq (C:_{R}M)$
or $\ a_{\alpha }I_{g}K_{\lambda }\subseteq C$ or $a_{\alpha
}J_{h}K_{\lambda }\subseteq C,$ it follows that either $a_{\alpha
}r_{g}s_{h}\in (C:_{R}M)$ or $a_{\alpha }r_{g}t_{\lambda }\in (C:_{R}M)$ or $%
a_{\alpha }s_{h}t_{\lambda }\in (C:_{R}M).$ Therefore, $(C:_{R}M)$ is a
graded $A$-2-absorbing ideal of $R$.
\end{proof}

 %_----------------------------------------------------------------------------------------------------------
%-------------------------Theorem 2. 9 -----------------------------------------------------------------------

\begin{theorem}
Let $R$ be a $G$-graded ring, $M$ a graded $R$-module,
$C$ a graded submodule of $M$ and $A\subseteq h(R)$ be a
multiplicatively closed subset of $R$ with $(C:_{R}M)\bigcap A=\emptyset .$
Then the following statements are equivalent:
 \begin{enumerate}[\upshape (i)]
   \item $C$ is a graded $A$-2-absorbing submodule of $M$;
   \item There is a fixed $a_{\alpha }\in A$ such that for every $r_{g},$ $%
s_{g}\in h(R),$ we have either $(C:_{M}a_{\alpha
}^{2}r_{g}s_{h})=(C:_{M}a_{\alpha }^{2}r_{g})$ or $(C:_{M}a_{\alpha
}^{2}r_{g}s_{h})=(C:_{M}a_{\alpha }^{2}s_{h})$ or $(C:_{M}a_{\alpha
}^{3}r_{g}s_{h})=M.$
 \end{enumerate}
\end{theorem}
\begin{proof}
$(i)\Rightarrow (ii)$Assume that $C$ is a graded $A$-2-absorbing
submodule of $M$. Then there exists a fixed $a_{\alpha }\in A$ such that
whenever $r_{g}s_{h}m_{\lambda }\in C,$ where $r_{g},$ $s_{g}\in h(R)$ and
$m_{\lambda }\in h(M),$ then either $a_{\alpha }r_{g}s_{h}\in (C:_{R}M)$ or $%
a_{\alpha }r_{g}m_{\lambda }\in C$ or $a_{\alpha }s_{h}m_{\lambda }\in C.$
Now let $m_{\lambda }\in (C:_{M}a_{\alpha }^{2}r_{g}s_{h})\cap h(M).$ Hence $%
(a_{\alpha }r_{g})(a_{\alpha }s_{h})m_{\lambda }\in C.$ Then either $%
a_{\alpha }^{2}r_{g}m_{\lambda }\in C$ or $a_{\alpha }^{2}s_{h}m_{\lambda
}\in C$ or $a_{\alpha }^{3}r_{g}s_{h}\in (C:_{R}M)$ as $C$ is a graded $A$%
-2-absorbing submodule of $M.$ If $a_{\alpha }^{2}r_{g}m_{\lambda }\in C$ or $a_{\alpha
}^{2}s_{h}m_{\lambda }\in C,$ then $(C:_{M}a_{\alpha
}^{2}r_{g}s_{h})\subseteq (C:_{M}a_{\alpha }^{2}r_{g})\cup (C:_{M}a_{\alpha
}^{2}s_{h}).$ Clearly $(C:_{M}a_{\alpha }^{2}r_{g})\cup (C:_{M}a_{\alpha
}^{2}s_{h})\subseteq (C:_{M}a_{\alpha }^{2}r_{g}s_{h}).$ So $%
(C:_{M}a_{\alpha }^{2}r_{g})\cup (C:_{M}a_{\alpha
}^{2}s_{h})=(C:_{M}a_{\alpha }^{2}r_{g}s_{h}).$ By \cite[Lemma 2.2]{15}, $(C:_{M}a_{\alpha
}^{2}r_{g})=(C:_{M}a_{\alpha }^{2}r_{g}s_{h})$ or $\ (C:_{M}a_{\alpha
}^{2}s_{h})=(C:_{M}a_{\alpha }^{2}r_{g}s_{h}).$ If $a_{\alpha }^{3}r_{g}s_{h}\in (C:_{R}M)$, then $(C:_{M}a_{\alpha
}^{3}r_{g}s_{h})=M.$

\bigskip

$(ii)\Rightarrow (i)$ Let $\ r_{g}s_{h}m_{\lambda }\in C,$ where $r_{g},$ $%
s_{g}\in h(R)$ and $m_{\lambda }\in h(M).$ Thus $m_{\lambda }\in
(C:_{M}a_{\alpha }^{2}r_{g}s_{h}).$ By given hypothesis, we have $%
(C:_{M}a_{\alpha }^{2}r_{g}s_{h})=(C:_{M}a_{\alpha }^{2}r_{g})$ or $%
(C:_{M}a_{\alpha }^{2}r_{g}s_{h})=(C:_{M}a_{\alpha }^{2}s_{h})$ or $%
(C:_{M}a_{\alpha }^{3}r_{g}s_{h})=M.$ Then $a_{\alpha }^{2}r_{g}m_{\lambda
}\in C$ or $a_{\alpha }^{2}s_{h}m_{\lambda }\in C$ or $a_{\alpha
}^{3}r_{g}s_{h}\in (C:_{R}M),$ this yields that either $a_{\alpha
}^{3}r_{g}m_{\lambda }\in C$ or $a_{\alpha }^{3}s_{h}m_{\lambda }\in C$ or $%
a_{\alpha }^{3}r_{g}s_{h}\in (C:_{R}M).$ By setting $s^{\ast }=a_{\alpha
}^{3},$ $C$ is a graded $A$-2-absorbing submodule of $M$.
\end{proof}
  %_----------------------------------------------------------------------------------------------------------
%-------------------------Lemma 2.10  -----------------------------------------------------------------------

\begin{lemma}
Let $R$ be a $G$-graded ring, $M$ a graded $R$-module, $A\subseteq h(R)$ be
a multiplicatively closed subset of $R\ $and $C$ a graded $A$-2-absorbing
submodule of $M$. Then the following statements hold:
\begin{enumerate}[\upshape (i)]
   \item There exists a fixed $a_{\alpha }\in A$ such that $(C:_{M}a_{\alpha
}^{3})=(C:_{M}a_{\alpha }^{n})$ for all $n\geqslant 3.$
   \item There exists a fixed $a_{\alpha }\in A$ such that $(C:_{R}a_{\alpha
}^{3}M)=(C:_{R}a_{\alpha }^{n}M)$ for all $n\geqslant 3.$
 \end{enumerate}

\end{lemma}
\begin{proof}

$(i)$ Since $C$ a graded $A$-2-absorbing submodule of $M,$ there exists a
fixed $a_{\alpha }\in A$ such that whenever $r_{g}s_{h}m_{\lambda }\in C,$
where $r_{g},s_{h}\in h(R)$ and $m_{\lambda }\in h(M),$ then either $%
a_{\alpha }r_{g}s_{h}\in (C:_{R}M)$ or $a_{\alpha }r_{g}m_{\lambda }\in C$
or $a_{\alpha }s_{h}m_{\lambda }\in C.$ Let $m_{\lambda }\in
(C:_{M}a_{\alpha }^{4})\cap h(M)$, it follows that $a_{\alpha
}^{4}m_{\lambda }=a_{\alpha }^{2}(a_{\alpha }^{2}m_{\lambda })\in C.$ Then $%
a_{\alpha }^{3}m_{\lambda }\in C$ as $C$ is a graded $A$-2-absorbing
submodule of $M,$ it follows that $m_{\lambda }\in (C:_{M}a_{\alpha }^{3}).$
Hence $(C:_{M}a_{\alpha }^{4})\subseteq (C:_{M}a_{\alpha }^{3}).$ Since the
other inclusion is always satisfied, we get $(C:_{M}a_{\alpha
}^{4})=(C:_{M}a_{\alpha }^{3}).$ Assume that $(C:_{M}a_{\alpha
}^{3})=(C:_{M}a_{\alpha }^{k})$ for all $k<n.$ We will show that $%
(C:_{M}a_{\alpha }^{3})=(C:_{M}a_{\alpha }^{n}).$ Let $m_{\lambda }^{\prime
}\in (C:_{M}a_{\alpha }^{n})\cap h(M)$, it follows that $a_{\alpha
}^{n}m_{\lambda }^{\prime }=a_{\alpha }^{2}(a_{\alpha }^{n-2}m_{\lambda
}^{\prime })\in C.$ Then either $a_{\alpha }^{3}m_{\lambda }^{\prime }\in C$
or $a_{\alpha }^{n-1}m_{\lambda }^{\prime }\in C$ as $C$ is a graded $A$%
-2-absorbing submodule of $M,$ it follows that $m_{\lambda }^{\prime }\in
(C:_{M}a_{\alpha }^{3})\cup (C:_{M}a_{\alpha }^{n-1})=(C:_{M}a_{\alpha
}^{3}) $ by induction hypothesis. Therefore $(C:_{M}a_{\alpha
}^{3})=(C:_{M}a_{\alpha }^{n})$ for every $n\geqslant 3.$
\bigskip

$(ii)$ Follows directly from $(i)$.
\end{proof}

  %_----------------------------------------------------------------------------------------------------------
%-------------------------Theorem 2.11  -----------------------------------------------------------------------

\begin{theorem}
Let $R$ be a $G$-graded ring, $M$ a graded $R$-module, $A\subseteq h(R)$ be
a multiplicatively closed subset of $R$ and $C$ a graded submodule of $%
M$ with $(C:_{R}M)\cap A=\emptyset $. Then the following statements are
equivalent:
 \begin{enumerate}[\upshape (i)]
   \item $C$ is a graded $A$-2-absorbing submodule.
   \item $(C:_{M}a_{\alpha })$ is a graded 2-absorbing submodule of $M$ for some $%
a_{\alpha }\in A.$
 \end{enumerate}
\end{theorem}
\begin{proof}
 $(i)\Rightarrow (ii)$ Assume that $C$ is a graded $A$-2-absorbing submodule. Then there exists a
fixed $a_{\alpha }\in A$ such that whenever $r_{g}s_{h}m_{\lambda }\in C,$
where $r_{g},s_{h}\in h(R)\ $and $m_{\lambda }\in h(M),$ then either $%
a_{\alpha }r_{g}s_{h}\in (C:_{R}M)$ or $a_{\alpha }r_{g}m_{\lambda }\in C$
or $\ a_{\alpha }s_{h}m_{\lambda }\in C.$ By Lemma 2.10, we have $%
(C:_{M}a_{\alpha }^{3})=(C:_{M}a_{\alpha }^{n})$ \ and $(C:_{R}a_{\alpha
}^{3}M)=(C:_{R}a_{\alpha }^{n}M)$ for all $n\geqslant 3.$ We show that $%
(C:_{M}a_{\alpha }^{6})=(C:_{M}a_{\alpha }^{3})$ is a graded 2-absorbing
submodule of $M$. Let $r_{g}s_{h}m_{\lambda }\in (C:_{M}a_{\alpha }^{6})\
\ $for some $r_{g},s_{h}\in h(R)\ $and $m_{\lambda }\in h(M).$ It follows
that, $a_{\alpha }^{6}(r_{g}s_{h}m_{\lambda })=(a_{\alpha
}^{2}r_{g})(a_{\alpha }^{2}s_{h})(a_{\alpha }^{2}m_{\lambda })\in C.$ Then
either $a_{\alpha }(a_{\alpha }^{2}r_{g})(a_{\alpha }^{2}s_{h})=a_{\alpha
}^{5}r_{g}s_{h}\in (C:_{R}M)$ or $a_{\alpha }(a_{\alpha
}^{2}r_{g})(a_{\alpha }^{2}m_{\lambda })=a_{\alpha }^{5}r_{g}m_{\lambda }\in
C$ or $a_{\alpha }(a_{\alpha }^{2}s_{h})(a_{\alpha }^{2}m_{\lambda
})=a_{\alpha }^{5}s_{h}m_{\lambda }\in C$ as $C$ is a graded $A$-2-absorbing
submodule of $M$. It follows that either $r_{g}s_{h}\in (C:_{R}a_{\alpha
}^{5}M)=(C:_{R}a_{\alpha }^{6}M)=((C:_{M}a_{\alpha }^{6}M):_{R}M)$ or $%
r_{g}m_{\lambda }\in (C:_{M}a_{\alpha }^{5}M)=(C:_{M}a_{\alpha }^{6}M)$ or $%
s_{h}m_{\lambda }\in (C:_{M}a_{\alpha }^{5}M)=(C:_{M}a_{\alpha }^{6}M).$
Thus $(C:_{M}a_{\alpha }^{6})$ is a graded 2-absorbing submodule of $M.$

\bigskip

$(ii)\Rightarrow (i)$ Assume that $(C:_{M}a_{\alpha })$ is a graded
2-absorbing submodule of $M$ for some $a_{\alpha }\in A.$ Let $\
r_{g}s_{h}m_{\lambda }\in C\subseteq (C:_{M}a_{\alpha }),$ where $r_{g},$ $%
s_{g}\in h(R)$ and $m_{\lambda }\in h(M).$ Then $r_{g}s_{h}m_{\lambda }\in
(C:_{M}a_{\alpha })$. Since $(C:_{M}a_{\alpha })$ is a graded 2-absorbing
submodule of $M,$ we get either $r_{g}s_{h}\in ((C:_{M}a_{\alpha }):_{R}M)$
or $r_{g}m_{\lambda }\in (C:_{M}a_{\alpha })$ or $s_{h}m_{\lambda }\in
(C:_{M}a_{\alpha })$. Thus $a_{\alpha }r_{g}s_{h}\in (C:_{R}M)$ or $%
a_{\alpha }r_{g}m_{\lambda }\in C$ or $a_{\alpha }s_{h}m_{\lambda }\in C.$
Therefore, $C$ is a graded $A$-2-absorbing submodule.
\end{proof}

 %----------------------------------------------------
 %_----------------------------------------------------------------------------------------------------------
%-------------------------Theorem 2.12  -----------------------------------------------------------------------
Let $M$ and $M^{\prime }$ be two graded $R$-modules. A
homomorphism of graded $R$-modules $f:M\rightarrow M^{\prime }$ is a
homomorphism of $R$-modules verifying $f(M_{g})\subseteq M_{g}^{\prime }$
for every $g\in G$, (see \cite{23}).

The following result studies the behavior of graded $A$-$2$-absorbing
submodules under graded homomorphism.
\begin{theorem}
Let $R$ be a $G$-graded ring and $M,$ $%
M^{\prime }$ be two graded $R$-modules and $f:M\rightarrow M^{\prime }$ be a
graded homorphism. Let $A\subseteq h(R)$ be a multiplicatively closed subset
of $R.$
 \begin{enumerate}[\upshape (i)]

   \item If $C^{\prime }$ is a graded $A$-2-absorbing submodule of $%
M^{\prime }\ $and $(f^{-1}(C^{\prime }):_{R}M)\cap A=$ $\emptyset ,$ then $%
f^{-1}(C^{\prime })$ is a graded $A$-2-Absorbing submodule of $M.$
   \item If $f$ is a graded epimorphism and $C$ is a graded $A$%
-2-absorbing submodule of $M$ with $Kerf\subseteq C,$ then $f(C)$ is a
graded $A$-2-absorbing submodule of $M^{\prime }$.
 \end{enumerate}

\end{theorem}

\begin{proof}
$(i)$ Assume that $C^{\prime }$ is a graded $A$-2-absorbing submodule of $%
M^{\prime }$. Now, let $r_{g},s_{h}\in h(R)$ and $m_{\lambda }\in h(M)$ such
that $r_{g}s_{h}m_{\lambda }\in f^{-1}(C^{\prime }).$ Hence $%
f(r_{g}s_{h}m_{\lambda })$ $=r_{g}s_{h}f(m_{\lambda })\in C^{\prime }.$ Since $C^{\prime }$ is a graded $A$-2-absorbing submodule, there exists $%
a_{\alpha }\in A$ such that either $a_{\alpha }r_{g}s_{h}\in (C^{\prime
}:_{R}M^{\prime })$ or $a_{\alpha }r_{g}f(m_{\lambda })=f(a_{\alpha
}r_{g}m_{\lambda })\in C^{\prime }$ or $a_{\alpha }s_{h}f(m_{\lambda
})=f(a_{\alpha }s_{h}m_{\lambda })\in C^{\prime }.$ It follows that either $%
a_{\alpha }r_{g}s_{h}\in (C^{\prime }:_{R}M^{\prime })\subseteq
(f^{-1}(C^{\prime }):_{R}M)$ or $a_{\alpha }r_{g}m_{\lambda }\in
f^{-1}(C^{\prime })$ or $a_{\alpha }s_{h}m_{\lambda }\in f^{-1}(C^{\prime
}). $ Therefore, $f^{-1}(C^{\prime })$ is a graded $A$-2-absorbing submodule
of $M$.

\bigskip

$(ii)$ Assume that $C$ is a graded $A$-2-absorbing submodule of $M$
containing $Kerf$. First, we want to show that $(f(C):_{R}M^{\prime
})\bigcap A=\emptyset .$ Suppose on the contrary that there exists $a_{g}\in
(f(C):_{R}M^{\prime })\bigcap A.$ Hence $a_{g}M^{\prime }\subseteq f(C),$
this implies that $f(a_{g}M)=a_{g}f(M)\subseteq a_{g}M^{\prime }\subseteq
f(C).$ It follows that, $a_{g}M\subseteq a_{g}M+Kerf\subseteq C+Kerf=C.$
Hence $a_{g}M\subseteq C$ and so, $a_{g}\in (C:_{R}M),$ which is a
contradiction since $(C:_{R}M)\bigcap A=\Phi .$ Now, let $%
r_{g}s_{h}m_{\lambda }^{\prime }\in f(C)$ for some $r_{g},$ $s_{h}\in h(R)$
and $m_{\lambda }^{\prime }\in h(M^{\prime }).$ Then, there exists $c_{\beta
}\in C\bigcap h(M)$ such that $r_{g}s_{h}m_{\lambda }^{\prime }=f(c_{\beta
}).$ Since $f$ is a graded epimorphism and $m_{\lambda }^{\prime }\in
h(M^{\prime })$, there exists $m_{\lambda }\in h(M)$ such that $m_{\lambda
}^{\prime }=f(m_{\lambda }).$ Then $f(c_{\beta })=r_{g}s_{h}m_{\lambda
}^{\prime }=r_{g}s_{h}f(m_{\lambda })=f(r_{g}s_{h}m_{\lambda })$, and so $%
c_{\beta }-r_{g}s_{h}m_{\lambda }\in Kerf\subseteq C$, it follows that $%
r_{g}s_{h}m_{\lambda }\in C.$ Since $C$ is a graded $A$-2-absorbing
submodule of $M$, there exists $a_{\alpha }\in A$ such that $a_{\alpha
}r_{g}s_{h}\in (C:_{R}M)$ or $a_{\alpha }r_{g}m_{\lambda }\in C$ or $%
a_{\alpha }s_{h}m_{\lambda }\in C.$ Then we have either $a_{\alpha
}r_{g}s_{h}\in (C:_{R}M)\subseteq (f(C):_{R}M^{\prime })$ or $\ a_{\alpha
}r_{g}m_{\lambda }^{\prime }=a_{\alpha }r_{g}f(m_{\lambda })=f(a_{\alpha
}r_{g}m_{\lambda })\in f(C)$ or $a_{\alpha }s_{h}m_{\lambda }^{\prime
}=a_{\alpha }s_{h}f(m_{\lambda })=f(a_{\alpha }s_{h}m_{\lambda })\in f(C).$
Thus $f(C)$ is a graded $A$-2-absorbing submodule of $M^{\prime }.$
\end{proof}

 %----------------------------------------------------
 %_----------------------------------------------------------------------------------------------------------
%-------------------------Example 2.13  -----------------------------------------------------------------------
Let $R$ be a $G$-graded ring, $M$ a graded $R$-module, $A\subseteq h(R)$ a
multiplicatively closed subset of $R$ and $C$ a graded submodule of $M$ with
$(C:_{R}M)\bigcap A=\emptyset.$ We say that $C$ is \textit{a graded }$A$%
\textit{-prime submodule of }$M$ if there exists a fixed $a_{\alpha }\in A$
and whenever $r_{g}m_{\lambda }\in C\ $where $r_{g}\in h(R)$ and $m_{\lambda
}\in h(M),$ implies that either $a_{\alpha }r_{g}\in (C:_{R}M)$ or $a_{\alpha
}m_{\lambda }\in C$ (see \cite{5}).

It is easy to see that every graded $A$-prime submodule
of $M$ is a graded $A$-$2$-absorbing submodule. The following example shows
that the converse is not true in general.
\begin{example}
Let $G=%
%TCIMACRO{\U{2124} }%
%BeginExpansion
\mathbb{Z}
%EndExpansion
_{2}$ and $R=%
%TCIMACRO{\U{2124} }%
%BeginExpansion
\mathbb{Z}
%EndExpansion
$ be a $G$-graded ring with $R_{0}=%
%TCIMACRO{\U{2124} }%
%BeginExpansion
\mathbb{Z}
%EndExpansion
$ and $R_{1}=\{0\}.$ Let $M=%
%TCIMACRO{\U{2124} }%
%BeginExpansion
\mathbb{Z}
%EndExpansion
_{6}$ \ be a graded $R$-module with $M_{0}=%
%TCIMACRO{\U{2124} }%
%BeginExpansion
\mathbb{Z}
%EndExpansion
_{6}\ $ and $M_{1}=\{\bar{0}\}.$ Now, consider the graded submodule $C=\{%
\bar{0}\}\ $of $M,$ then $C$ is not a graded prime submodule since $2\cdot
\bar{3}\in C\ $ where $2\in R_{0}$ and $\bar{3}\in M_{0}$ but $\bar{3}\notin
C\ $and $2\notin (C:_{%
%TCIMACRO{\U{2124} }%
%BeginExpansion
\mathbb{Z}
%EndExpansion
\ }M)=6%
%TCIMACRO{\U{2124} }%
%BeginExpansion
\mathbb{Z}
%EndExpansion
.$ However an easy computation shows that $C$ is a graded 2-absorbing
submodule of $M$. Now let $A$ be the set of units in $R.$ Then $C$ is a
graded $A$-2-absorbing of $M$ but not graded $A$-prime submodule.
\end{example}

%------------------------------------------------------------------------------
%-------------------------------Theorem 2.14----------------------------------
\begin{theorem}
Let $R$ be a $G$-graded ring, $M$ a graded $R$-module
and $A\subseteq h(R)$ be a multiplicatively closed subset of $R.$ Then the
intersection of two graded $A$-prime submodule is a graded $A$-2-absorbing.
\end{theorem}

\begin{proof}

Let $C_{1}$ and $C_{2}$ be two graded $A$-prime submodules of $M$ and $%
C=C_{1}\cap C_{1}.$ Let $r_{g}s_{h}m_{\lambda }\in C$ for some $%
r_{g},s_{h}\in h(R)$ and $m_{\lambda }\in h(M).$ Since $C_{1}$ is a graded $%
A $-prime submodule of $M$ \ and $r_{g}(s_{h}m_{\lambda })\in C_{1},$ there
exists $a_{1_{\alpha }}\in A$ such that $a_{1_{\alpha }}r_{g}\in
(C_{1}:_{R}M)$ or $a_{1_{\alpha }}s_{h}m_{\lambda }\in C_{1}.$ If $%
a_{1_{\alpha }}s_{h}m_{\lambda }=s_{h}(a_{1_{\alpha }}m_{\lambda })\in
C_{1}, $ then either $a_{1_{\alpha }}s_{h}\in (C_{1}:_{R}M)$ or $%
a_{1_{\alpha }}^{2}s_{h}\in C_{1}$ since $C_{1}$ is a graded $A$-prime
submodule and hence either $a_{1_{\alpha }}s_{h}\in (C_{1}:_{R}M)$ or $%
a_{1_{\alpha }}s_{h}\in C_{1}$ by \cite[Lemma 2.10]{5}. In a similar manner, since $C_{2}$ is a graded $A$-prime submodule of $M$ \
and $r_{g}s_{h}m_{\lambda }\in C_{2},$ there exists $a_{2_{\alpha }}\in A$
such that $a_{2_{\alpha }}r_{g}\in (C_{1}:_{R}M)$ or $a_{2_{\alpha
}}s_{h}\in (C_{1}:_{R}M)$ $\ $or $a_{2_{\alpha }}m_{\lambda }\in C_{2}.$ Now
put $a_{\beta }=$ $a_{1_{\alpha }}a_{2_{\alpha }}\in A.$ Then either $%
a_{\beta }r_{g}s_{h}\in (C:_{R}M)$ or $a_{\beta }r_{g}\in C$ or $a_{\beta
}s_{h}\in C.$ Therefore, $C$ is a graded 2-absorbing submodule of $M$.
\end{proof}

%\section*{\bf Acknowledgements}The authors wish to thank sincerely the referees for their valuable comments and suggestions.

\bigskip\bigskip\bigskip\bigskip

\end{document}